\documentclass{article}
\usepackage{amsmath,amssymb,graphicx,stmaryrd}
\usepackage{amsthm,mathrsfs}
\usepackage{ascmac}
\usepackage[all]{xy}
\usepackage{comment}

  \title{The Calabi invariant and The Lyndon-Hochschild-Serre spectral sequence}
  \author{TAKUYA FUJITANI}
  \date{\today}

\makeatletter

  \@addtoreset{equation}{section}

  \theoremstyle{plain}
  \newtheorem{axm}{公理}[section]
  \newtheorem{thm}[axm]{Theorem}
  \newtheorem{prp}[axm]{Proposition}
  
  \newtheorem{cor}[axm]{Corollary}
  
  \newtheorem*{clm*}{主張}

  \theoremstyle{definition}
  \newtheorem{dfn}[axm]{Definition}

  \theoremstyle{remark}
  \newtheorem*{note*}{記号}
  \newtheorem*{rmk*}{注釈}
  \newtheorem*{exm*}{例}
  \newtheorem{exm}[axm]{Example}
  \newtheorem{rmk}[axm]{Remark}

  \def\Z{\mathbb{Z}}
  
  \def\R{\mathbb{R}}

  \def\id{\mathrm{id}}
  \def\sgn{\mathrm{sgn}}

  \def\Hom{\mathrm{Hom}}

  \newcommand{\Rm}[1]{\mbox{$\mathrm{#1}$}}

\begin{document}

\maketitle

\begin{abstract}
Let $G$ be a group and $N$ be a normal subgroup of $G$. 
There exists the group extension $G$ of $G/N$ by $N$. 
For a $G$-module $A$ which $N$ acts on trivially and a $G$-invariant homomorphism on $N$ to $A$, we obtain a central extension of $G/N$ by $A$.
By using connection cochains, we exhibit the formula of its extension class such that clarify the relation among connection cochains, extension classes and the LHS spectral sequence. 
\end{abstract}

\tableofcontents

% !TEX root=Cal_spec_main.tex

\section{Introduction}

Let $D$ be a closed unit disk in $\R^{2}$ with a standard symplectic form $\omega$, and $G = \Rm{Symp}(D, \omega)$ denote the group of symplectomorphisms on $D$. 
There exists a group extension
\[
1 \to G_{\Rm{rel}} \to G \to G_{\partial} \to 1, 
\]
where we set $G_{\Rm{rel}} = \{g \in \Rm{Symp}(D, \omega) \mid g|_{\partial D} = \id\}$ and $G_{\partial} = \Rm{Diff}^{+}(\partial D)$ the group of diffeomorphisms on the circle $\partial D$. 
On such the group $G_{\Rm{rel}}$, we have a homomorphism $\Rm{Cal} : G_{\Rm{rel}} \to \R$ which is called the {\it Calabi invariant}. 
Moriyoshi \cite{Moriyoshi2016} defined a group cochain $\tau : G \to \R$ such that, for $g \in G$ and $h \in G_{\Rm{rel}}$, $\tau(h g) = \Rm{Cal}(h) + \tau(g)$, and proved that the negative coboundary $-\delta\tau$ induces the extension class $e(G/K) \in H^{2}(G_{\partial}; \R)$ of the Calabi extension 
\[
0 \to G_{\Rm{rel}} \to G/K \to G_{\partial} \to 1. 
\]
Here we set $K = \Rm{Ker}\Rm{Cal}$. 
Moriyoshi called this cochain $\tau$ the connection cochain over $\Rm{Cal}$. 

In the general case, let $G$ be a group, $N$ be a normal subgroup of $G$ and $A$ be a $G$-module which $N$ acts on trivially. 
Let $f : N \to A$ be a homomorphism such that $f(g n g^{-1}) = f(n)$ for $g \in G$ and $n \in N$. 
There are a group extension 
\[
1 \to N \to G \to G/N \to 1
\]
and the induced central extension 
\[
0 \to A \to G_{A} \to G/N \to 1, 
\]
where $G_{A}$ is the quotient of $G \times A$ by the equivalence relation $(n g, a) \sim (g, f(n) + a)$ for $(g, a) \in G \times A$ and $n \in N$. 
If we have a connection cochain $\tau : G \to A$ over $f : N \to A$, that is, $\tau(n g) = f(n) + \tau(g)$ for $g \in G$ and $n \in N$, then the negative coboundary $-\delta\tau$ induces the extension class $e(G_{A}) \in H^{2}(G/N; A)$ of the above central extension. 
To be more precise, it is stated as follows. 

\begin{thm}[see Theorem \ref{thm:1}]
Let $e(G_{A})$ denote the extension class of $G$ with coefficients in $A$. 
Then there exists a connection cochain $\tau : G \to A$ over $f$ and we can regard a 2-cochain $-\delta\tau$ on $G$ as a 2-cocycle on $G/N$ whose cohomology class coincides $e(G_{A})$. 
\end{thm}

In the term of the spectral sequence, we can regard this theorem as follow: 
For the group extension $1 \to N \to G \to G/N \to 1$, there is the Lyndon-Hochschild-Serre (LHS) spectral sequence $E_{2}^{p, q} \Rightarrow H^{p+q}(G)$. 
Then $E_{2}^{1, 0}$ is isomorphic to $\Hom(N, A)^{G}$ that is the group of homomorphisms on $N$ to $A$ such that $f(g n g^{-1}) = f(n)$ for $g \in G$ and $n \in N$. 
For $f \in \Hom(N, A)^{N}$, a connection cochain $\tau$ over $f$ is a representation cocycle of an element in $E_{2}^{1, 0}$ corresponding to $f \in \Hom(N, A)^{G}$. 
The transgression image $d_{2}f \in H^{2}(G/N; A)$ of $f \in \Hom(N, A)^{G}$ represents $\delta\tau$. 
On the other hand, the transgression image $d_{2}f$ coincides the coupling of the extension class $e(G/N')$ of $0 \to N/N' \to G/N' \to G/N \to 1$ that is the Abelization of $1 \to N \to G \to G/N \to 1$ and the negative $f$. 
This coupling is just the extension class $e(G_{A}) \in H^{2}(G/N; A)$, that is, $d_{2}f = -e(G_{A})$.  
It follows that $e(G_{A}) = -d_{2}f = [-\delta\tau]$. 

This paper is organized as follows. 
In section 2, we recall the theory of the cohomology for groups. 
It is known that the low-dimensional cohomology group for groups has the another definition. 
There is a bijection between the second cohomology group for groups and group extensions, in particularly, central extensions. 
In section 3, we briefly review a relation between the Calabi extension and the universal Euler class of the group of orientation-preserving diffeomorphisms on the circle due to Moriyoshi. 
In section 4, we introduce connection cochains over some homomorphisms and prove that its negative coboundary represents the extension class for central extensions. 
In the final section, we recall the LHS spectral sequence. 
We understand our theorem in the terms of the spectral sequence and the five-term exact sequence. 

\section*{Acknowledgement}
I would like to express my deepest gratitude to Prof. Moriyoshi who provided carefully considered feedback and valuable comments.
% !TEX root=Cal_spec_main.tex

\section{The group cohomology and the group extensions}

Let $G$ be a group and $A$ be an Abelian group. 
Assume that $A$ has a left and $G$-action that is written by $g.a \in A$ for $a \in A$ and $g \in G$. 

\begin{dfn}
The {\it group cochain complex} $(C^{\ast}(G; A), \delta)$ with coefficients in $A$ is given by the pair 
\[
C^{p}(G; A) = \{c : G^{p} \to A\}, \quad \delta : C^{p}(G; A) \to C^{p+1}(G; A), 
\]
where a {\it group} $p$-{\it cochain} $c : G^{p} \to A$ is an arbitrary function on the $p$-tuple product of $G$ and where the coboundary $\delta$ is defined to be 
\begin{align*}
\delta c(g_{1}, \ldots, g_{p+1}) = g_{1}.c(g_{2}, \ldots, g_{p+1}) & + \sum_{i=1}^{p}(-1)^{i}c(g_{1}, \ldots, g_{i} g_{i+1}, \ldots, g_{p+1})\\
& + (-1)^{p+1} c(g_{1}, \ldots, g_{p}). 
\end{align*}
Denote by $H^{\ast}(G; A)$ the cohomology group of $C^{\ast}(G; A)$, called the {\it group cohomology} of $G$ with coefficients in $A$. . 
\end{dfn}

Denote by $C^{p}(G; A)_{N}$ the group of the {\it normalized} group $p$-cochains $c : G^{p} \to A$ which are group $p$-cochains such that $c(g_{1}, \ldots, g_{p}) = 0$ whenever one of the $g_{i}$ is equal to $1$. 
As can be seen easily, if $c$ is a normalized cochain then $\delta c$ is so. 
Then $(C^{p}(G; A)_{N}, \delta)$ is a subcomplex of $C^{\ast}(G; A)$.  
As is well known, the cohomology of the normalized group cochain complex $C^{\ast}(G; A)_{N}$ coincides the one of the ordinary group cochain complex $C^{\ast}(G; A)$. 

\begin{exm}
Let $G$ act on $A$ trivially. 
Then every group 1-cocyle $f  \in C^{1}(G; A)$ satisfies that 
\[
\delta f(g, h) = f(h) - f(g h) + f(g) = 0
\]
for $g, h \in G$, so that, $f$ is a homomorphism. 
On the other hand, for any 0-cochain $c \in C^{0}(G; A)$, the coboundary is $\delta c(g) = g \cdot c - c = 0$ for $g \in G$. 
This implies that all of the group 1-coboundaries are trivial. 
Then the first cohomology group $H^{1}(G; A)$ is just the group of homomorphisms $\Hom(G, A)$. 
\end{exm}

As with the first cohomology, there is an alternative definition for the second cohomology groups. 
We recall group extensions and central extensions. 

\begin{dfn}
Let $A$, $\Gamma$ and $G$ be arbitrary groups. 
A {\it group extension} $\Gamma$ of $G$ by $A$ is a short exact sequence of groups: 
\[
1 \to A \to \Gamma \to G \to 1. 
\]
Furthermore, $\Gamma$ is a {\it central extension} if $A \to \Gamma$ factors through the center of $\Gamma$. 
\end{dfn}

Consider the group extension $\Gamma$ of $G$ by an Abelian group $A$: 
\[
0 \to A \hookrightarrow\ \Gamma \to G \to 1
\]
Set $s : G \to \Gamma$ an arbitrary set-theorical section. 
There is a group action of $G$ on $A$ defined by 
\[
g \cdot a = s(g) a s(g)^{-1}
\]
for $g \in G$ and $a \in A$ since $p(g \cdot a) = g p(a) g^{-1} = 1$, that is, $g \cdot a \in A$. 
In fact, this group action is independent of the choice of section $s$. 

For $g, h \in G$, we have $s(g) s(h) s(g h)^{-1} \in A$ because of $p(s(g) s(h) s(g h)^{-1}) = 1$. 
Then we can define a group 2-cochain $\chi \in C^{2}(G; A)$ as the following: 
\[
\chi(g, h) = s(g) s(h) s(g h)^{-1} \in A. 
\]
It is easy to check that $\chi \in C^{2}(G; A)$ is a group 2-cocycle and the cohomology class $[\chi] \in H^{2}(G; A)$ is independent on the choice of section $s : G \to \Gamma$, so that, the class depends only on the group extension. 

\begin{dfn}
The {\it extension class} of the group extension $\Gamma$, denote by $e(\Gamma)$, is the cohomology class $[\chi] \in H^{2}(G; A)$. 
\end{dfn}

If a section $s : G \to \Gamma$ is a homomorphism then the corresponding 2-cocycle $\chi \in C^{2}(G; A)$ is obviously trivial. 
This implies that the splitting extension has the trivial extension class. 
The following is well known. 

\begin{prp}[see Brown \cite{Brown1982}]
Let $G$ be a group and $A$ be a $G$-module. 
The second cohomology group $H^{2}(G; A)$ is to the equivalence classes of central extensions of $G$ by $A$; 
\[
H^{2}(G; A) \cong \{\mbox{central extensions of $G$ by $A$}\}/\{\mbox{splitting extensions}\}. \qed
\]
\end{prp}

\begin{rmk}
Let $G$ be a group and $N$ be a non-Abelian normal subgroup of $G$. 
Then we have a group extension $1 \to N \to G \xrightarrow{p} G/N \to 1$. 
Since $N$ is not Abelian, we cannot define the extension class of such an extension. 
Set $N'$ the commutator subgroup of $N$. 
For $n, m \in N$ and $g \in G$, we have $g [n, m] g^{-1} = [g n g^{-1}, g m g^{-1}]$, where $[ \ , \ ]$ is the commutator on $N$. 
This implies that $N'$ is normal in $G$ and we can define a new group extension: 
\[
0 \to N/N' \to G/N' \to G/N \to 1. 
\]
Since $N/N'$ is Abelian, we obtain the extension class $e(G/N') \in H^{2}(G/N; N/N')$. 

Moreover, let $s : G/N \to G$ be a set-theorical section of the projection $p$. 
We have a map $\chi : G/N \times G/N \to N$ defined by $\chi(g, h) = s(x) s(y) s(x y)^{-1} \in N$ for $x, y \in G/N$ and we can regard $\chi$ as a representation cocycle of the extension class of $e(G/N')$ by composing the natural projection $N \to N/N'$. 
\end{rmk}

\begin{rmk}
Let $1 \to N \xrightarrow{i} \Gamma \to G \to 1$ be a group extension of $G$ by $N$ and $f : N \to A$ be a $\Gamma$-homomorphism, that is, $f(\gamma \cdot n) = \gamma \cdot f(n)$ for $\gamma \in \Gamma$ and $n \in N$, where $A$ is an Abelian group. 
We can extend the group extension by $N$ to the central extension by $A$: 
\[
0 \to A \to \Gamma_{A} \to G \to 1, 
\]
where $\Gamma_{A} = \Gamma \times A/\sim$ with the equivalence relation $(i(n) \gamma, a) \sim (\gamma, f(n) + a)$ for $(\gamma, a) \in \Gamma \times A$ and $n \in N$. 
It is easy to check that the extension class $e(\Gamma_{A})$ is the push-forward of the extension class of $e(\Gamma/[N, N])$ by $f$, and the class is called the extension class with coefficients in $A$ along $f$. 
\end{rmk}

\begin{exm}[the universal covering of the circle diffeomorphisms.]
\label{exm:univ_cov}
Set $\Rm{Diff}^{+}(S^{1})$ the group of orientation-preserving diffeomorphisms on the circle $S^{1}$. 
Let $H$ be the universal covering group of $\Rm{Diff}^{+}(S^{1})$. 
Then we obtain a central extension $0 \to \Z \to H \xrightarrow{\pi} \Rm{Diff}^{+}(S^{1}) \to 1$. 
By the inclusion $\Z \hookrightarrow \R$, we extend this central extension to the central extension by $\R$: 
\[
0 \to \R \to H_{\R} \to G_{\partial} \to 1. 
\]
As is well known, the extension class $e(H_{\R})$ is given by the formula 
\[
\chi(g_{1}, g_{2}) = \frac{1}{4\pi^{2}} \int_{0}^{2\pi}\left(h_{1} \circ h_{2}(x) - h_{1}(x) - h_{2}(x)\right) \, dx
\]
with $h_{1}, h_{2} \in H$ and $g_{1} = \pi(h_{1})$, $g_{2} = \pi(g_{2})$. 
\end{exm}

% !TEX root = Cal_spec_main.tex

\section{The Calabi invariants and the Euler class}

In this section, we review the work on the relation between the Calabi invariant and extension classes of group extensions due to Moriyoshi \cite{Moriyoshi2016}. 
Let $D = \{(x,y) \in \R^{2} \mid x^2+y^2 \leq 1\}$ be a closed unit disk in $\R^{2}$ with a standard symplectic form $\omega = dx \wedge dy$, and let $G = \Rm{Symp}(D)$ denote the group of symplectomorphisms on $D$. 
There exists the following short exact sequence of groups:
\[
1 \to G_{\Rm{rel}} \to G \to G_{\partial} \to 1.
\]
Here we set $G_{\Rm{rel}} = \{g \in G \mid g|_{\partial D} = \id\}$ and $G_{\partial} = \Rm{Diff}^{+}(\partial D)$. 

\begin{dfn}[Tsuboi \cite{Tsuboi00}]
The {\it Calabi invariant} $\Rm{Cal} : G_{\Rm{rel}} \to \R$ is defined 
\[
\Rm{Cal}(h) = - \int_{D} \eta \wedge h^{\ast}\eta
\]
for $h \in G_{\Rm{rel}}$, where $\eta$ is a 1-form on $D$ such that $d\eta = \omega$. 
\end{dfn}

The Calabi invariant $\Rm{Cal}$ is independent on the choice of $\eta$ such as $\omega = d\eta$; see McDuff-Salamon \cite{McDuff-Salamon98}. 

\begin{prp}
The Calabi invariant yeilds a $G$-invariant homomophism on $G_{\Rm{rel}}$.  
\end{prp}

\begin{proof}
Note that $d(\eta - g^{\ast}\eta) = \omega - g^{\ast}\omega = 0$ for $g \in G$, thus there exists a smooth function $f_{g} \in C^{\infty}(D)$ such that $\eta - g^{\ast}\eta = df_{g}$ since the first cohomology group $H^{1}(D)$ vanishes. 
We have $(\eta - h^{\ast}\eta)|_{\partial D} = 0$ because $h|_{\partial D} = \id$ for $h \in G_{\Rm{rel}} \subset G$ by definition. 
Then we obtain
\begin{align*}
\int_{D}h^{\ast}(\eta - {g}^{\ast}\eta) \wedge (\eta - {h}^{\ast}\eta) = & \int_{D}\{h^{\ast}\eta \wedge \eta - (gh)^{\ast}\eta \wedge \eta + h^{\ast}(g^{\ast}\eta \wedge \eta)\}\\
= & \int_{D}h^{\ast}\eta \wedge \eta - \int_{D}(gh)^{\ast}\eta \wedge \eta + \int_{D}g^{\ast}\eta \wedge \eta\\
= & -\Rm{Cal}(h) + \Rm{Cal}(gh) - \Rm{Cal}(h)
\end{align*}
for $g, h \in G_{\Rm{rel}}$.

On the other hand, we obtain 
\begin{align*}
\int_{D} {h}^{\ast}(\eta - {g}^{\ast}\eta) \wedge (\eta - {h}^{\ast}\eta) = & \int_{D} h^{\ast}df_{g} \wedge (\eta - g^{\ast}\eta)\\
= & \int_{D} d(h^{\ast}f_{g}(\eta - g^{\ast}\eta))\\
= & \int_{\partial D} h^{\ast}f_{g}(\eta - g^{\ast}\eta) = 0
\end{align*}
for $g, h \in G_{\Rm{rel}}$. 
These imply that the Calabi invariant $\Rm{Cal} : G_{\Rm{rel}} \to \R$ is a homomorphism. 

Furthermore, let $g \in G$ and $h \in G_{\Rm{rel}}$. 
Then we have
\begin{align*}
g^{\ast}\eta \wedge (gh)^{\ast}\eta = & (\eta - df_{g}) \wedge h^{\ast}(\eta - df_{g})\\
= & \eta \wedge h^{\ast}\eta - df_{g} \wedge h^{\ast}\eta - \eta \wedge h^{\ast}df_{g} + df_{g} \wedge h^{\ast}df_{g}\\
= & \eta \wedge h^{\ast}\eta - d(f_{g} \wedge h^{\ast}\eta) + f_{g} \omega + d(\eta \wedge h^{\ast}f_{g}) - h^{\ast}f_{g}\omega + d(f_{g} \wedge h^{\ast}df_{g}). 
\end{align*}
Note that $h|_{\partial D} = \id$, thus we obtain 
\begin{align*}
\Rm{Cal}(ghg^{-1}) = & -\int_{D}\eta \wedge (ghg^{-1})^{\ast}\eta\\
= & -\int_{D}g^{\ast}\eta \wedge (gh)^{\ast}\eta\\
= & -\int_{D}\eta \wedge h^{\ast}\eta - d(f_{g} \wedge h^{\ast}\eta) + d(\eta \wedge h^{\ast}f_{g}) + d(f_{g} \wedge h^{\ast}df_{g})\\
= & -\int_{D}\eta \wedge h^{\ast}\eta - \int_{\partial D}f_{g} \, \eta + \int_{\partial D} \eta \, f_{g} + \int_{\partial D}f_{g} \wedge df_{g}\\
= & \Rm{Cal}(h). 
\end{align*}
This implies that $\Rm{Cal}$ is $G$-invariant. 
\end{proof}

It is possible to extend the Calabi invariant from $G_{\Rm{rel}}$ to $G$ by the same formula. 
Namely, there is a cochain  $\tau : G \to \R$ defined by
\[
\tau(g) = - \int_{D} \eta \wedge g^{\ast}\eta
\]
for $g \in G$. 
However, this cochain $\tau$ depends on the choice of 1-form $\eta$ such that $\omega = d\eta$. 

\begin{prp}
For $g \in G$ and $h \in G_{\Rm{rel}}$, we have
\[
\tau(hg) = \Rm{Cal}(h) + \tau(g). 
\]
\end{prp}

\begin{proof}
See Moriyoshi \cite{Moriyoshi2016}. 
\end{proof}

We call this cochain $\tau : G \to \R$ as a {\it connection cochain} over $\Rm{Cal}$.

\begin{prp}
There is a group 2-cocycle $\sigma : G_{\partial} \times G_{\partial} \to \R$ such that 
\[
\sigma(g|_{\partial D}, g|_{\partial D}) = - \delta\tau(g, h)
\]
for $g, h \in G$. 
Furthermore, the cohomology class of $\sigma$ is the extension class of the central extension:
\[
0 \to \R \to G/K \to G_{\partial} \to 1, 
\]
where $K = \Rm{Ker}\Rm{Cal}$ and we identify $\R$ as $G_{\Rm{rel}}/K$. 
\end{prp}

\begin{proof}
See Moriyoshi \cite{Moriyoshi2016}. 
\end{proof}

This central extension $0 \to \R \to G/K \to G_{\partial} \to 1$ is called the Calabi extension. 

Recall that the universal covering group $H$ of $G_{\partial}$ and central extensions described in the last section, example \ref{exm:univ_cov}. 

\begin{thm}[Moriyoshi \cite{Moriyoshi2016}] 
In the above terms, the extension class $e(G/K)$ coincides with $e(H_{\R}) \in H^{2}(G_{\partial})$. 
\end{thm}

%\input{conn_LHS}
% !TEX root = Cal_spec_main.tex

\section{The connection cochains and the extension classes}

Let $G$ be a discrete group and $N$ be a normal subgroup of $G$. 
There exists a group extension: 
\[
1 \to N \to G \xrightarrow{p} G/N \to 1. 
\]
Set $N'$ the commutator subgroup of $N$.  
We have a new group extension by the Abelian group: 
\[
0 \to N/N' \to G/N' \to G/N \to 1. 
\]

Let $A$ be a $G$-module which $N$ acts on trivially, and let $f : N \to A$ be a $G$-invariant homomorphism. 
Then we extend the group extension by $N$ to the central extension
\[
0 \to A \to G_{A} \to G/N \to 1, 
\]
where $G_{A}$ is the quotient of $G \times A$ with the equivalence relation $(n g, a) \sim (g, f(n) + a)$. 

\begin{dfn}
A cochain $\tau : G \to A$ that satisfies the condition 
\[
\tau(ng) = f(n) + \tau(g)
\]
for $n \in N$ and $g \in G$ is called a {\it connection cochain} over $f$. 
\end{dfn}

\begin{thm}
\label{thm:1}
In the context of the above, there exists a connection cochain $\tau : G \to A$ over $f$. 

Furthermore, there exists a 2-cocycle $\sigma$ on $G/N$ in $A$ which holds the following: 
\begin{enumerate}
\item 
We have $\sigma(p(g), p(h)) = -\delta\tau(g, h)$ for $g, h \in G$.

\item 
The cohomology class $[\sigma] \in H^{2}(G/N; A)$ is independent of the choice of connection cochain $\tau$ over $f$. 

\item 
The extension class $e(G_{A})$ of the central extension $0 \to A \to G_{A} \to G/N \to 1$ coincides with $[\sigma] \in H^{2}(G/N; A)$. 

\item 
Moreover, $e(G_{A})$ equals to the push-forward of $e(G/N') \in H^{2}(G/N; N/N')$ by $f : N \to A$. 
Here, $e(G/N')$ is the extension class of the group extension $0 \to N/N' \to G/N' \to G/N \to 1$. 

\end{enumerate}
\end{thm}

\begin{proof}
Let $s : G/N \to G$ be a set-theorical section of the natural projection $p : G \to G/N$. 
Since $p(g (s \circ p(g))^{-1}) = p(g) p(g)^{-1} = 1$ for $g \in G$, there is a cochain $\tau_{s} : G \to A$ defined by 
\[
\tau_{s}(g) = f(g \iota(g)^{-1}), \quad \iota(g) = s \circ p(g). 
\] 
Then we obtain 
\[
\tau_{s}(ng) = f(ng \iota(ng)^{-1}) = f(n) + f(g \iota(g)^{-1}) = f(n) + \tau_{s}(g)
\]
for $g \in G$ and $n \in N$. 
Therefore the cochain $\tau_{s}$ is a connection cochain over $f$. 

We have $g \cdot f(n) = f(\iota(g) n \iota(g)^{-1})$ for $g \in G$ and $n \in N$ since $N$ acts on $A$ trivially. 
Thus we obtain
\begin{align*}
\delta\tau_{s}(g, h) = & g \cdot \tau_{s}(h) - \tau_{s}(g h) + \tau_{s}(g)\\
= & g \cdot f(h \iota(h)^{-1}) - f(g h \iota(g h)^{-1}) + f(g \iota(g)^{-1})\\
= & -\left(-f(\iota(g) h \iota(h)^{-1} \iota(g)^{-1}) - f(g \iota(g)^{-1}) + f(g h \iota(g h)^{-1})\right)\\
= & -f(\iota(g) \iota(h) h^{-1} \iota(g)^{-1} \iota(g) g^{-1} g h \iota(g h)^{-1})\\
= & -f(\iota(g) \iota(h) \iota(g h)^{-1})
\end{align*}
for $g, h \in G$. 
Then $\iota(g) \iota(h) \iota(g h)^{-1} = s(p(g)) s(p(h)) s(p(g) p(h))^{-1}$ induces the extension class of $0 \to N/N' \to G/N' \to G/N \to 1$. 
This implies $[-\delta \tau_{s}] = f_{\ast}e(G/N')$. 

Here, we have a natural commutative diagram: 
\[\xymatrix{
1 \ar[r] & N \ar@{^{(}->}[r] \ar[d]_{f} & G \ar[d]^{\pi} \ar[r]^{p} & G/N\ar@{}[d]|{\rotatebox{90}{=}} \ar[r] & 1\\
0 \ar[r] & A \ar@{^{(}->}[r] & G_{A} = G \times A/\sim \ar[r]^-{p'} & G/N \ar[r] & 1, 
}\]
where $\pi : G \to G_{A}$ is the inclusion. 
Then $s' = \pi \circ s$ is a set-theorical section of $p'$. 
Therefore we can write that $\delta\tau_{s}(g, h) = -s'(p(g)) s'(p(h)) s'(p(g) p(h))^{-1}$ and this implies that $\delta\tau_{s}$ induces the extension class of the central extension $0 \to A \to G_{A} \to G/N \to 1$. 

Finally, we prove that the cohomology class $[\sigma] \in H^{2}(G/N; A)$ is independent on the choice of connection cochain over $f$. 

Let $\tau$ be an arbitrary connection cochain over $f$. 
We have
\[
\delta\tau(g, h) = g \cdot \tau(h) - \tau(g h) + \tau(g)
\]
for $g, h \in G$. 
Note that $N$ acts on $A$ trivially, thus we obtain 
\begin{align*}
\delta\tau(n g, h) = & n g \cdot \tau(h) - \tau(n g h) + \tau(n g)\\
 = & g \cdot \tau(h) - f(n) - \tau(g h) + f(n) + \tau(g) = \delta\tau(g, h)
\end{align*}
and 
\begin{align*}
\delta\tau(g, n h) = & g \cdot \tau(n h) - \tau(g n h) + \tau(g)\\
= & g \cdot f(n) + g \cdot \tau(h) - f(g n g^{-1}) - \tau(g h) + \tau(g) = \delta\tau(g, h)
\end{align*}
for $g, h \in G$ and $n \in N$. 
These imply that $\delta\tau$ depends only on $G/N$. 
Therefore there exists a 2-cocycle $\sigma$ on $G/N$ such that $\sigma(p(g), p(h)) = \delta\tau(g, h)$. 

Let $\tau$ and $\tau'$ be connection cochains over $f$, and let $\sigma$ and $\sigma'$ be the 2-cocycles on $G/N$ corresponding to $\tau$ and $\tau'$, respectively. 
Since we have 
\[
(\tau - \tau')(ng) = f(n) + \tau(g) - f(n) - \tau'(g) = (\tau - \tau')(g)
\]
for $g \in G$ and $n \in N$, there exists a group 1-cochain $T$ on $G/N$ such that $T(p(g)) = -(\tau - \tau')(g)$. 
This implies $\sigma - \sigma' = \delta T$, so that, $\sigma$ and $\sigma'$ define the same cohomology class $[\sigma] = [\sigma'] \in H^{2}(G/N; A)$. 
\end{proof}
% !TEX root = Cal_spec_main.tex

\section{The Lyndon-Hochschild-Serre spectral sequence and the five-term exact sequence}
Hereinafter, we consider group cochain complexes as normalized ones.  

Let $G$ be a group, $N$ be a normal subgroup of $G$ and $A$ be a $G$-module. 
Then there is the group extension:
\[
1 \to N \to G \to G/N \to 1. 
\]

We define a filtration $(C_{p}^{\ast})_{p\in\Z}$ of $C^{\ast}(G; A)$ as follows: 
\[
\begin{cases}
C_{p}^{q} = C^{q}(G; A), & \mbox{if $p \leq 0$}, \\
C_{p}^{q} = 0, & \mbox{if $q < p$} 
\end{cases}
\]
and for $0 < p \leq q$, we set $C_{p}^{q}$ the group of $q$-cochains $c \in C^{q}(G; A)$ such that $f(g_{1}, \ldots, g_{p}) = 0$ whenever $(n-j+1)$ of the arguments belong to $N$. 
We can check $\delta(C_{p}^{\ast}) \subset C_{p}^{\ast}$ easily. 

This filtration defines a spectral sequence $E_{r}^{p, q}$, called the {\it Lyndon-Hochschild-Serre spectral sequence} for the group extension: 
\[
1 \to N \to G \to G/N \to 1. 
\] 

\begin{dfn}[Hochschild-Serre \cite{Hochschild-Serre53}]
Let $G$ be a group, $N$ be a normal subgroup of $G$ and $A$ be a $G$-module. 
We set $Z_{r}^{p, q} = \{ c \in C_{p}^{p+q} \mid \delta c \in C_{p+r}^{p+q+1}\}$. 
The Lyndon-Hochschild-Serre (LHS) spectral sequence is a spectral sequence defined as
\[
E_{r}^{p, q} = Z_{r}^{p, q}/(Z_{r-1}^{p-1, q+1} + \delta Z_{r-1}^{p+1-r, q+r-2})
\]
and differential operators $d_{r} : E_{r}^{p, q} \to E_{r}^{p+r, q-r+1}$ are induced by $\delta$. 
\end{dfn}

\begin{prp}[Hochschild-Serre \cite{Hochschild-Serre53}]
The LHS spectral sequence $E_{r}^{p, q}$ satisfies the follows. 
\begin{enumerate}
\item 
There exists an isomorphism 
\[
E_{1}^{p, q} \cong C^{p}(G/N; H^{q}(N; A)). 
\]
and the differential $d_{1}$ corresponds to the group coboundary map $\delta_{G/N} : C^{p}(G/N; H^{q}(N; A)) \to C^{p+1}(G/N; H^{q}(N; A))$. 
\item 
There exists an isomorphism
\[
E_{2}^{p, q} \cong H^{p}(G/N; H^{q}(N; A)). 
\]
\item 
This spectral sequence $E_{r}^{p, q}$ converges to $H^{p+q}(G; A)$, that is, there is an isomorphism
\[
E_{\infty}^{p, q} \cong \Rm{Im}(H^{p+q}(C_{p}) \to H^{p+q}(G; A))/\Rm{Im}(H^{p+q}(C_{p+1}) \to H^{p+q}(G; A)). 
\]
\end{enumerate}
\end{prp}

\begin{proof}[Sketch of proof]
Suppose $c \in C^{p+q}(G; A)$. 

Let $\sigma$ be a $(p, q)$-shuffle, that is, $\sigma$ is a permutation of $(1, \ldots, p+q)$ such that $\sigma(1) < \cdots \sigma(p)$ and $\sigma(p+1) < \cdots < \sigma(p+q)$. 
We define $c_{\sigma} \in C^{p+q}(G; A)$ by
\[
c_{\sigma}(\alpha_{1}, \ldots, \alpha_{q}, \beta_{1}, \ldots, \beta_{p}) = c(g_{1}, \ldots, g_{p+q}), 
\]
where
\[
g_{\sigma(i)} = 
\begin{cases}
\beta_{i} & \mbox{if $1 \leq i \leq p$}, \\
\beta_{\sigma(i) + p - i}^{-1} \cdots \beta_{1}^{-1} \alpha_{i-p} \beta_{1} \cdots \beta_{\sigma(i) + p - i}, & \mbox{if $p+1 \leq i \leq p+q$}
\end{cases}\]
and 
\[
c_{p} = \sum_{\sigma} \sgn(\sigma) c_{\sigma}. 
\]

Finally, we define a homomorphism $\varphi : C^{p+q}(G; A) \to C^{p}(G/N; C^{q}(N; A))$ such that 
\[
\varphi(c)(\beta_{1}, \ldots, \beta_{p})(\alpha_{1}, \ldots, \alpha_{q}) = c_{p}(\alpha_{1}, \ldots. \alpha_{q}, \beta_{1}, \ldots, \beta_{p})
\]
for $\beta_{1}, \ldots, \beta_{p} \in G/N$ and $\alpha_{1}, \ldots, \alpha_{q} \in N$. 
This homomorphism $\varphi$ induces isomorphisms $E_{1}^{p, q} \to C^{p}(G/N; H^{q}(N; A))$ and $E_{2}^{p, q} \to H^{p}(G/N; H^{q}(N; A))$. 
\end{proof}

Denote by $A^{N}$ the subgroup of $A$ consisting of all of $G$-invariant elements. 
It is easy to check that $H^{0}(N; A)$ is isomorphic to $A^{N}$. 
For a group $p$-cocycle $c \in C^{p}(G/N; A^{N})$, we define a new cocycle $\overline{c} \in C^{p}(G/N; A)$ by composing with natural projections $G \to G/N$ and a natural inclusion $A^{N} \hookrightarrow A$; 
\[
\overline{c} : G^{p} \to (G/N)^{p} \xrightarrow{c} A^{N} \hookrightarrow A. 
\]
This correspondence define a homomorphism $\Rm{inf} : H^{p}(G/N; A^{N}) \to H^{p}(G; A)$, which is called the {\it inflation homomorphism}. 

On the other hand, an inclusion map $N \hookrightarrow G$ induces a homomorphism $\Rm{res} : H^{q}(G; A) \to H^{q}(N; A)$. 
It is easy to check that $\Rm{res}$ factors through the $G$-invariant part of $H^{q}(N, A)$. 
This homomorphism $\Rm{res} : H^{q}(G; A) \to H^{0}(G/N; H^{q}(N; A))$ is called the {\it restriction homomorphism}. 

\begin{prp}[Hochschild-Serre \cite{Hochschild-Serre53}]
Let $1 \to N \to G \to G/N \to 1$ be a group extension and $A$ be a $G$-module. 
Then there exists a group exact sequence
\begin{align*}
0 \to H^{1}(G/A; A^{N}) \xrightarrow{\Rm{inf}} H^{1}(G; A) \xrightarrow{\Rm{res}} & H^{0}(G/N; H^{1}(N; A))\\
& \xrightarrow{d_{2}} H^{2}(G/N; A^{N}) \xrightarrow{\Rm{inf}} H^{2}(G; A). 
\end{align*}
\qed
\end{prp}

This exact sequence is called the {\it five-term exact sequence} or the {\it inflation-restriction exact sequence} for a group extension $1 \to N \to G \to G/N \to 1$. 

In general, given the spectral sequence converging to some cohomology group, it is well known that there is an exact sequence called the five-term exact sequence. 

Suppose that $N$ acts on $A$ trivially. 
According to Hochschild-Serre \cite{Hochschild-Serre53}, the transgression map on the LHS spectral sequence is given by a cup product. 

We define a cup product $\cup : \Hom(N, A) \times N/N' \to A$ by $f \cup n = f(n) \in A$. 
This cup product induces a cup product $\cup : \Hom(N, A)^{G} \times H^{2}(G/N; N/N') \to H^{2}(G/N; A)$, where $\Hom(N, A)^{G}$ denotes a group of $G$-invariant homomorphisms on $N$ to $A$. 

\begin{thm}[Hochschild-Serre \cite{Hochschild-Serre53}]
Let $G$ be a group, $N$ be a normal subgroup and $A$ be a $G$-module which $N$ acts on trivially. 
Set $N'$ be the commutator subgroup of $N$. 
For $f \in H^{1}(N; A)^{G}$, the transgression image of $f$ can be written by $d_{2}f = -f \cup e(G/N') \in H^{2}(G/N; A)$ explicitly, where $e(G/N') \in H^{2}(G/N; N/N')$ is the extension class of the group extension $0 \to N/N' \to G/N' \to G/N \to 1$. 
\end{thm}

In the term of the LHS spectral sequence, a connection cochain $\tau : G \to A$ over a $G$-invariant homomorphism $f : N \to A$ is a representation cocycle of the class of $E_{2}^{1, 0}$ which corresponds to $f \in \Hom(N, A)^{G} = H^{0}(G/N; H^{1}(N; A))$. 
Thus $\delta\tau$ represents $d_{2}f$ in $E_{2}^{0, 2} \cong H^{2}(G/N; A)$. 
By definition, the cup product $f \cup e(G/N')$ coincides with the push-forward of $e(G/N')$ by $f$, so that, $f \cup e(G/N') = e(G_{A})$.
Summarizing the above, we conclude the follow: 
\[
e(G_{A}) = f \cup e(G/N') = -d_{2}f = [-\delta\tau]. 
\]
This is an another proof of theorem \ref{thm:1}. 

\begin{cor}
Let $D$ be a closed unit disk in $\R^{2}$ with a standard symplectic form $\omega$ and $G = \Rm{Symp}(D, \omega)$ denote the group of symplectomorphisms on $D$. 
We set $G_{\Rm{rel}} = \{ g \in G \mid g|_{\partial D} = \id\}$ and $G_{\partial} = \Rm{Diff}^{+}(\partial D)$. 
There are a group extension $1 \to G_{\Rm{rel}} \to G \to G_{\partial} \to 1$ and the Calabi invariant $\Rm{Cal} : G_{\Rm{rel}} \to \R$. 
Let $H$ be the universal covering group of $G_{\partial}$ thus there is a central extension $0 \to \Z \to H \to G_{\partial} \to 1$. 

Then we obtain $\pi^{2}e(H_{\R}) = -d_{2}\Rm{Cal}$. 
\end{cor}

In other words, in the five-term exact sequence for $1 \to G_{\Rm{rel}} \to G \to G_{\partial} \to 1$, $\Rm{Cal}$ is in $H^{1}(G_{\Rm{rel}}; \R)$ and $\pi^{2}e(H_{\R})$ is an image of $\Rm{Cal}$ in $H^{2}(G_{\partial}; \R)$: 
\begin{align*}
0 \to H^{1}(G_{\partial}; \R) \xrightarrow{\Rm{inf}} H^{1}(G; \R) \xrightarrow{\Rm{res}} & \Hom(G_{\Rm{rel}}, \R)^{G}\\
& \xrightarrow{d_{2}} H^{2}(G_{\partial}; \R) \xrightarrow{\Rm{inf}} H^{2}(G; \R). 
\end{align*}

\bibliography{reference}

\end{document}